\documentclass[11pt, reqno]{amsart}
\usepackage{amsmath,amssymb,mathtools}
\usepackage{microtype}
\usepackage[margin=1.08in]{geometry}
\usepackage{enumitem}
\usepackage{booktabs,longtable,array}
\usepackage{xcolor}
\usepackage[colorlinks=true,linkcolor=blue!55!black,citecolor=blue!55!black,urlcolor=blue!55!black]{hyperref}
\usepackage{fancyhdr}

\usepackage{aliascnt} 
\usepackage{cleveref} 

\numberwithin{equation}{section}

\theoremstyle{plain}
\newtheorem{theorem}{Theorem}[section]

\newaliascnt{lemma}{theorem}
\newtheorem{lemma}[lemma]{Lemma}
\aliascntresetthe{lemma}

\newaliascnt{proposition}{theorem}
\newtheorem{proposition}[proposition]{Proposition}
\aliascntresetthe{proposition}

\newaliascnt{corollary}{theorem}

\aliascntresetthe{corollary}

\theoremstyle{definition}
\newaliascnt{definition}{theorem}
\newtheorem{definition}[definition]{Definition}
\aliascntresetthe{definition}

\newaliascnt{remark}{theorem}
\newtheorem{remark}[remark]{Remark}
\aliascntresetthe{remark}

\newcommand{\Ric}{\operatorname{Ric}}

\newcommand{\Vol}{\operatorname{Vol}}

\newcommand{\R}{\mathbb{R}}

\newcommand{\dd}{\,\mathrm{d}}

\author{Guosheng Jiang}
\address[Guosheng Jiang]{School of Mathematics, Shandong University, Jinan 250100, China}
\email{gsjiang@sdu.edu.cn}

\author{Mingxiang Li}
\address[Mingxiang Li]{Department  of Applied  Mathematics, The Hong Kong Polytechnic University, Hong Kong, China}
\email{mingxiang.li@polyu.edu.hk}

\author{Zhehui Wang}
\address[Zhehui Wang]{School of Sciences, Great Bay University, Dongguan 523000, China}
\email{wangzhehui@gbu.edu.cn}

\title{A Cheng-Yau type estimate for positive biharmonic functions}
\begin{document}

\begin{abstract}
We establish a Cheng--Yau type estimate for positive biharmonic
functions on complete Riemannian manifolds with Ricci curvature satisfies $\Ric_g \ge -(n-1)K g$. If $u$ is a positive biharmonic function in $B_{2R}(p)$, then
$$
-\frac{\Delta_g u}{u}
+\frac{1}{8n}\frac{|\nabla u|_g^2}{u^2}
\le C_n\left(R^{-2}+K\right)
\quad\text{on }B_R(p).
$$
Further, if the Ricci curvature is nonnegative, every global positive biharmonic
function satisfies $\Delta_g u\equiv c$ and the sharp estimate
$|\nabla u|_g^2\le2cu$ for some nonnegative constant $c$. We also show that every positive
$k$-polyharmonic function has nonnegative constant
$(k-1)$-st Laplacian and growth of order at most $2k-2$.

\end{abstract}
\maketitle

\section{Introduction}

The classical Liouville theorem plays an important role in analysis, which asserts that every positive harmonic function on $\mathbb{R}^n$ must be a constant. In the 1970s, Yau \cite{Yau} extended such classical result to complete manifolds with nonnegative Ricci curvature.
\begin{theorem}[Yau's theorem]\label{Yau's theorem}
    Let $(M^n,g)$ be a complete Riemannian manifold with nonnegative Ricci curvature. Each positive harmonic function  on $M^n$ must be a constant.
\end{theorem}

In the pioneering work of Yau \cite{Yau} and the subsequent work of Cheng-Yau \cite{Cheng-Yau} and Li-Yau \cite{LiYau},  the gradient estimate they developed on manifolds has become a fundamental tool in geometric analysis. Yau posed the problem of whether the linear space of polynomial growth harmonic functions has finite dimension (see \cite{Y3, Y5, Y4}, and a survey in \cite{L1}). Specifically, for a fixed nonnegative constant $d$, define
\begin{equation*}
    \mathcal{H}^d(M)=\left\{u:\, \Delta_g u=0,\ \ |u(x)|\le C(1+d_g(x,p))^d\right\}
\end{equation*}
for some point $p\in M^n$ and some constant $C>0$. Under the assumption that the Ricci curvature is nonnegative, Yau conjectured that $\dim \mathcal{H}^d(M)<+\infty$ for all $d\ge 0$. This problem was eventually resolved by Li and Tam \cite{Li-Tam} for $n=2$ (another proof see \cite{DF}) and by Colding and Minicozzi \cite{CM97} for $n\ge 3$. Moreover, the harmonic functions on manifolds have a significant impact on the study of the structure of complete open manifolds (see \cite{CCM, CM-CPAM, Li-Tam-92-JDG}).  Lin and Zhang \cite{linzh2019} studied the linear space of polynomial-growth ancient solutions of heat equations on manifolds with nonnegative Ricci curvature (See \cite{rowi1959,Widder} for related discussions on Euclidean space). Later, Colding and Minicozzi \cite{coldmin2021} obtained the sharp bounds for the dimension of this linear space.

For polyharmonic functions on Euclidean space, however, the situation is quite different. First, there do exist positive non-constant polyharmonic functions; for example, for $x=(x_1,\ldots,x_n)\in\mathbb R^n$, the function $x_1^2+1$ satisfies 
\begin{equation}\label{nonconstant-exm}
    \Delta^k (x_1^2+1)\equiv 0
\end{equation}
 for any integer $k\ge 2$. It is natural to ask whether one can obtain some control and characterization of positive polyharmonic functions. In 1966, Kuran \cite{Kuran} showed that any positive polyharmonic function $u$ (i.e., $\Delta^k u=0$) must be a positive polynomial of degree at most $2k-2$. This result provides a fairly comprehensive characterization of positive polyharmonic functions.

 Polyharmonic functions play a significant role in higher-order conformal geometry on conformally flat manifolds, particularly in the study of $Q$-curvature (see \cite{Martinazzi}). More  applications  can be found the monograph \cite{GGS}. The conformally covariant GJMS operator (known as the Paneitz operator in the fourth-order case) is a higher-order operator. The leading term of the GJMS operator $P_{2k,g}$ is $(-\Delta_g)^k$. When the manifold is Ricci-flat (see \cite{Gover}), one has
$$
P_{2k,g}=(-\Delta_g)^k.
$$
Further discussions on the GJMS operator can be found in the survey \cite{CG}.

On the other hand, biharmonic functions on manifolds arise naturally in the study of biharmonic maps between manifolds (see \cite{ChangWangYang, MonOni}). The regularity theory for biharmonic maps is now well developed, with substantial results obtained in, e.g., \cite{ChangWangYang, Struwe, Wang}. In contrast, Liouville-type theorems for biharmonic functions on manifolds appear to be comparatively less explored. Recently, Wang--Zhu \cite{WZ} and Bravo--Cortissoz \cite{BC1} investigated biharmonic functions of polynomial growth on manifolds with nonnegative Ricci curvature and established Liouville theorems. Subsequently, Bravo--Cortissoz \cite{BC2} extended this result to the polyharmonic setting. Wang--Zhu \cite{WZ} also proved the space of polynomial-growth biharmonic function is finite-dimensional. A key ingredient in these works is a Caccioppoli-type inequality (see \cite[Lemma 2.1]{WZ} and \cite[Corollary 8]{BC2}). Using this estimate together with a suitable correspondence between polyharmonic and harmonic functions, the authors are able to bound the dimension of the space of polynomial growth polyharmonic functions in terms of the known dimension bounds for harmonic functions of polynomial growth. As a direct consequence of these Liouville theorems, it follows that every polynomial growth kernel of the GJMS operator on Ricci-flat manifolds has finite dimension. Such kernels were previously studied in the context of conformally flat manifolds in \cite{Li}.

The polynomial growth assumption plays a pivotal role in the proofs of \cite{BC1, BC2, WZ}. This naturally prompts the question of whether one can establish analogous control for one-sided bounded biharmonic functions, in the spirit of Yau's theorem (see \Cref{Yau's theorem}). More precisely, we seek to determine what can be concluded for positive biharmonic functions on manifolds with nonnegative Ricci curvature. The example \eqref{nonconstant-exm} demonstrates that such functions are not necessarily constant; nonetheless, an analogue of Kuran's theorem \cite{Kuran} is expected to hold.  With this in mind, our first result is the following.
Throughout this paper, the manifold is assumed to be connected.

\begin{theorem}
\label{thm:sharp-poisson}
Let $(M^n,g)$ be a complete Riemannian manifold with nonnegative Ricci curvature.
Suppose that $u$ is a positive biharmonic function on $M^n$. Then,  $\Delta_g u$ must be a nonnegative constant $c$  and the gradient of $u$ is controlled by 
\begin{equation}
\label{eq:sharp-v}
  |\nabla u|_g^2\le 2cu.
\end{equation}
\end{theorem}
\begin{remark}
In particular, when $u$ is harmonic, \eqref{eq:sharp-v} forces $u$ to be constant, which is precisely Yau's theorem. Moreover, if $c\neq 0$, \eqref{eq:sharp-v} is equivalent to
$|\nabla \sqrt{u}|_g^2 \leq c/2,$
which in turn implies the quadratic growth upper bound of $u$ as follows
$$
u(x) \leq \frac{c}{2} d_g(x,p)^2 + Cd_g(x,p)+C
$$
for some fixed base point $p \in M^n$ and constant $C$. The example \eqref{nonconstant-exm} demonstrates that the coefficient of the quadratic term is sharp.
\end{remark}

 Motivated by the Cheng–Yau estimate \cite{Cheng-Yau},  we ask whether an analogous estimate can be obtained for positive biharmonic functions. In contrast to the harmonic case, the gradient estimate here involves the term 
$\Delta_g u$. Our second result in this paper is stated as follows.

\begin{theorem}
\label{thm:local-CY}
Let $(M^n,g)$ be a complete Riemannian manifold with 
$$
    \Ric_g\ge -(n-1)Kg
    \qquad\text{in }B_{2R}(p)\subset\subset M,
$$
 where $p\in M^n$, and $K\ge0$ and $R>0$ are constants.
Assume that $u\in C^4(B_{2R}(p))$ is a positive biharmonic  function in $B_{2R}(p)$.
Then, there exists some positive constant $C(n)$ depending only on $n$,   such that
\begin{equation}
\label{eq:local-main}
    -\frac{\Delta_g u}{u}
    +\frac{1}{8n}\cdot\frac{|\nabla u|_g^2}{u^2}
    \le
    C(n)\left(\frac1{R^2}+K\right)
    \qquad\text{in }B_R(p).
\end{equation}
In particular, we have 
\begin{equation}
\label{eq:local-gradient-root}
    \frac{|\nabla u|_g}{u}
    \le
    C(n)\left(
        \frac1R+\sqrt K
        +\sqrt{\frac{(\Delta_g u)_+}{u}}
    \right)
    \qquad\text{in }B_R(p).
\end{equation}
\end{theorem}

\begin{remark}
If $u$ is  harmonic (certainly biharmonic), the estimate
\eqref{eq:local-gradient-root} reduces to the classical  Cheng--Yau estimate
$$
    \frac{|\nabla u|_g}{u}
    \le C(n)\left(\frac1R+\sqrt K\right).
$$
Thus, \eqref{eq:local-main} can be viewed as a natural fourth-order analogue of the
Cheng--Yau estimate.
\end{remark}

Kuran's theorem \cite{Kuran} states that any positive polyharmonic function on Euclidean space must be a polynomial. Theorem \ref{thm:sharp-poisson} concerns with the biharmonic case. In fact, for positive polyharmonic functions on manifolds with nonnegative Ricci curvature, we can effectively control the growth of their lower-order derivatives. Our third result is stated as  follows with the convention $\Delta_g^0u=u$.

\begin{theorem}\label{thm:main}
    Let $(M^n,g)$ be a complete Riemannian manifold with nonnegative Ricci curvature. Let $u\in C^{2k}(M)$ be a positive polyharmonic function of order $k \ge 2$, namely $\Delta_g^k u \equiv 0$. Then the following assertions hold:
    \begin{enumerate}
        \item $\Delta_g^{k-1} u \equiv c_u$ for some constant $c_u \ge 0$; \label{constpoly}
        \item for each $0 \le i \le k-1$, there exist a point $p \in M^n$ and a constant $C>0$ such that
        $$
        |\Delta_g^i u(x)| \le C \bigl(1 + d_g(x,p)\bigr)^{2k-2i-2} \quad \text{for all } x \in M^n.
        $$\label{growth}
        \item for each $0\le j\le k-2$, there exist a point $p \in M^n$ and a constant $C>0$ such that
        $$
        |\nabla\Delta_g^j u(x)| \le C \bigl(1 + d_g(x,p)\bigr)^{2k-2j-3} \quad \text{for all } x \in M^n.
        $$\label{gradgrowth}
    \end{enumerate}
\end{theorem}
\begin{remark}
   Combining with the work of Wang--Zhu \cite{WZ} and Bravo--Cortissoz \cite{BC1, BC2}, the linear space spanned by positive polyharmonic functions of fixed order on manifolds with nonngeative Ricci curvature  has finite dimension.
\end{remark}

We conclude the introduction by briefly outlining the strategy and structure of the paper. Section \ref{sec2} recalls the parabolic Harnack inequality on manifolds satisfying the doubling property and the Neumann-Poincar\'e inequality, which allows us to establish the crucial Lemma \ref{lem:harnack-package} and Proposition \ref{thm:caloric-positivity}. Building upon these, we prove Theorem \ref{thm:main} in Section \ref{sec:polyhar}. Subsequently, we derive a Cheng–Yau type estimate for positive biharmonic functions and prove Theorem \ref{thm:local-CY}. Finally, applying Theorems \ref{thm:main} and \ref{thm:local-CY}, we obtain the sharp estimate \eqref{eq:sharp-v} and thereby complete the proof of Theorem \ref{thm:sharp-poisson}.

\subsection*{Disclosure on AI assistance.} The authors used AI-assisted tools, principally ChatGPT. The authors wrote and verified all theorem statements, proofs, and they take full responsibility for the contents of the paper.

\section{Preliminaries}\label{sec2}
In this section, we present several useful results that will be employed in the proofs of our main theorems. Here and below, $(M, g)$ is a complete Riemannian manifold. 

\subsection{Doubling property and Neumann-Poincar\'e inequality}\label{sec:pre}
We recall the definition of doubling property and uniform Neumann--Poincar\'e inequality.

\begin{definition}\label{def:VD}
We say that $M$ satisfies the \emph{doubling property} if there is a positive constant $C_D<\infty$ such that
\begin{equation}\label{eq:VD}
  \Vol(B_{2r}(x))\leq C_D\,\Vol(B_r(x))
  \qquad\text{for all }x\in M\text{ and }r>0.
\end{equation}
\end{definition}

\begin{definition}\label{def:PI}
We say that $M$ satisfies a \emph{uniform Neumann--Poincar\'e inequality} if there is a  positive constant $C_P<\infty$ such that
\begin{equation}\label{eq:PI}
  \int_{B_r(x)}|f-f_{B_r(x)}|^2\dd V
  \leq C_P r^2\int_{B_r(x)}|\nabla f|^2\dd V
\end{equation}
for all $x\in M$, $r>0$ and $f\in W^{1,2}_{\mathrm{loc}}(M)$, where
$$
  f_{B_r(x)}
  :=\frac{1}{\Vol(B_r(x))}\int_{B_r(x)}f\dd V.
$$
\end{definition}
We note that if $M$ has nonnegative Ricci curvature, then $M$ satisfies doubling property and uniform Neumann--Poincar\'e inequality.

\subsection{\bf Harnack inequality.}The equivalence between the \eqref{eq:VD}-\eqref{eq:PI} and scale invariant
parabolic Harnack inequalities is due to Grigor'yan \cite{Grigoryan1992} and Saloff-Coste \cite{SaloffCoste1992IMRN, SaloffCoste1992JDG}. The elliptic Harnack inequalities could be viewed as a direct consequence of the parabolic case by considering stationary heat solutions. For the readers' convenience, we first record
a standard form and then derive the precise consequences used in the
proof.

\begin{theorem}[Theorem 3.1 in \cite{SaloffCoste1992IMRN}]\label{thm:standard-PHI}
Let $(M,g)$ be a complete Riemannian manifold satisfying \eqref{eq:VD} and \eqref{eq:PI}, and fix $x_0\in M$. There are constants
$C_H$ and $C_E$ depending only on $C_D$ and $C_P$ with the following property:
\begin{enumerate}
\item Let $R>0$ and  $w>0$ be a
classical solution of 
$$\partial_t w=\Delta_g w, \qquad \text{ on } B_{2R}(x_0)\times (s,s+4R^2].
$$
Then
\begin{equation*}\label{eq:standard-PHI}
  \sup_{B_R(x_0)\times(s+R^2,s+2R^2)}w
  \leq
  C_H\inf_{B_R(x_0)\times(s+3R^2,s+4R^2)}w.
\end{equation*}
\item  If $h\geq0$ is harmonic in $B_{2R}(x_0)$, then
  \begin{equation*}\label{eq:EHI}
    \sup_{B_R(x_0)}h\leq C_E\inf_{B_R(x_0)}h.
  \end{equation*}
\end{enumerate}
\end{theorem}

Using the above theorem as a foundation, we derive the following lemma suited to our present context.
\begin{lemma}\label{lem:harnack-package}
Under the same assumption as in Theorem \ref{thm:standard-PHI}. The following estimates hold.
\begin{enumerate}

  \item Fix $0<a<1/4$.  There are constants
  $R_a$ and $C_a$ depending only on $a$, $C_D$ and $C_P$ such that, for every
  $b\in[1/2,1)$, if $w>0$ solves 
  $$\partial_tw=\Delta_g w \qquad \text{ on } B_{R_a}(x_0)\times(0,b],$$
  then
  \begin{equation*}\label{eq:one-point-PHI}
    w(x_0,a)\leq C_a w(x_0,b).
  \end{equation*}

  \item Fix two positive constants $0<\theta<\Theta$.  There is a constant
  $C$ depending only on $\theta$, $\Theta$, $C_D$ and $C_P$ such that every positive global heat solution satisfies
  \begin{equation*}\label{eq:two-point-PHI}
    w(x,\theta T)\leq Cw(y,\Theta T)
  \end{equation*}
  whenever $T>0$ and $d_g(x,y)^2\leq T$.
\end{enumerate}
\end{lemma}

\begin{proof}
We shall repeatedly use the following pointwise consequence of
Theorem~\ref{thm:standard-PHI}. If $u>0$ solves the heat equation on
$$
B_{2\rho}(z)\times(\sigma,\sigma+4\rho^2],
$$
and if
$$
z_-,z_+\in B_\rho(z),
\qquad
t_-\in(\sigma+\rho^2,\sigma+2\rho^2),
\qquad
t_+\in(\sigma+3\rho^2,\sigma+4\rho^2),
$$
then
\begin{equation}\label{eq:pointwise-parabolic-Harnack}
u(z_-,t_-)\le C_Hu(z_+,t_+).
\end{equation}

For part~\textup{(1)}, set
$$
g_b=b^{-1}g,
\qquad
W(z,\tau)=w(z,b\tau),
\qquad
\alpha=\frac{a}{b}.
$$
Then
$$
\partial_\tau W=\Delta_{g_b}W,
\qquad
\alpha\in[a,2a]\subset(0,\tfrac12).
$$
The constants of doubling property and uniform Neumann--Poincar\'e inequality are unchanged by this rescaling.
Take $R_a=1$, we note the function $W$ is defined on
$B^{g_b}_1(x_0)\times(0,1].$
Since $\alpha$ ranges in the fixed compact interval $[a,2a]$,
a finite parabolic Harnack chain, with a number of links depending
only on $a$, joins $(x_0,\alpha)$ to $(x_0,\frac12)$ inside
$B^{g_b}_1(x_0)\times(0,1]$. Hence
$$
W(x_0,\alpha)\le C_a'W(x_0,\tfrac12),
$$
where $C_a'$ depends only on $a$, $C_D$, and $C_P$.
A further fixed Harnack comparison, followed by continuity at
$\tau=1$, gives
$$
W(x_0,\tfrac12)\le C_HW(x_0,1).
$$
Consequently,
$$
w(x_0,a)=W(x_0,\alpha)
\le C_aW(x_0,1)
=C_aw(x_0,b),
$$
where $C_a$ depends only on $a$, $C_D$, and $C_P$.

For part~\textup{(2)}, set
$$
g_T=T^{-1}g,
\qquad
W(z,\tau)=w(z,T\tau).
$$
Then
$$
\partial_\tau W=\Delta_{g_T}W,
\qquad
d_{g_T}(x,y)\le1.
$$
Join $x$ to $y$ by a minimizing $g_T$-geodesic. By subdividing this
geodesic and the time interval $[\theta,\Theta]$ into sufficiently
many equal pieces, with the number of pieces depending only on
$\theta$ and $\Theta$, one obtains a finite parabolic Harnack chain
joining $(x,\theta)$ to $(y,\Theta)$. Since $w$ is a global positive
heat solution, every cylinder in the chain lies in its domain.
Iterating the Harnack inequality therefore yields
$$
W(x,\theta)\le CW(y,\Theta),
$$
where $C$ depends only on $\theta$, $\Theta$, $C_D$, and $C_P$.
Scaling back gives $w(x,\theta T)\le Cw(y,\Theta T)$.
\end{proof}

The next result is the core of the proof of Theorem \ref{thm:main}.  It uses  the parabolic Harnack inequality, the Ekeland's variational principle and the method developed  in \cite{coldmin2021,linzh2019}.

\begin{proposition}\label{thm:caloric-positivity}
Under the same assumption as in Theorem \ref{thm:standard-PHI},  let
$$
  U:M\times[0,\infty)\longrightarrow\R
$$
be a classical solution of
\begin{equation*}\label{eq:heat-U}
  \partial_tU=\Delta_g U.
\end{equation*}
Suppose that, for every $x\in M$, the map $t\mapsto U(x,t)$ is a polynomial of degree at most $m$, with the same $m$ for all $x$.  If
\begin{equation*}\label{eq:initial-positive}
  U(x,0)>0
  \qquad\text{for every }x\in M,
\end{equation*}
then
\begin{equation}\label{eq:caloric-positive-conclusion}
  U(x,t)>0
  \qquad\text{for every }x\in M\text{ and }t\geq0.
\end{equation}
\end{proposition}

\begin{proof}
The case $m=0$ is immediate.  Suppose $m\geq1$ and argue by contradiction. If there is $x_\star\in M$ and $ t_\star>0$ such that $U(x_\star, t_\star)\le 0$, then continuity in time gives a zero at a positive time.  For each $x\in M$, define the first zero time
\begin{equation*}\label{eq:tau-def}
  \tau(x)=\inf\{t>0:U(x,t)=0\}\in(0,\infty],
\end{equation*}
with the convention $\tau(x)=\infty$ if no zero occurs.  Since $U(x,0)>0$, every $\tau(x)$ is positive.  If $\tau(x)<\infty$, then
\begin{equation*}\label{eq:first-zero-properties}
  U(x,t)>0\quad\hbox{for $0\leq t<\tau(x)$},
  \qquad
  U(x,\tau(x))=0.
\end{equation*}
The function $\tau$ is lower semicontinuous. Set
$$
  F(x)=\sqrt{\tau(x)}\in(0,\infty].
$$
By the contradiction assumption it is finite somewhere.  Let
$$
  \alpha=\inf_M F\in[0,\infty).
$$
For every sufficiently large integer $i$, choose $z_i\in M$ with
$$
  F(z_i)\leq \alpha+i^{-2}.
$$
Applying the Ekeland's variational principle \cite{Ekeland1974} to $F$ on the
complete metric space $M$, 
it yields a point $x_i\in M$ such that $F(x_i)\leq F(z_i)<\infty$, $d_g(x_i,z_i)<\frac{1}{i}$, and
\begin{equation}\label{eq:Ekeland-main}
  F(y)\geq F(x_i)-\frac1i d_g(x_i,y)
  \qquad\text{for every }y\in M.
\end{equation}
Write
$$
  T_i=\tau(x_i),
  \qquad
  g_i=T_i^{-1}g,
  \qquad
  V_i(y,s)=U(y,T_i s).
$$
Then $V_i$ solves the heat equation with respect to $g_i$, and constants of doubling property and uniform Neumann-Poincar\'e inequality are unchanged by this rescaling.  Moreover,
\begin{equation*}\label{eq:center-zero}
  V_i(x_i,s)>0\quad(0\leq s<1),
  \qquad
  V_i(x_i,1)=0.
\end{equation*}

Let
$$
  \sigma_i=\left(1-\frac1{\sqrt i}\right)^2.
$$
If $d_{g_i}(x_i,y)\leq \sqrt{i}$, then $d_g(x_i,y)\leq\sqrt{T_i}\sqrt i$, and \eqref{eq:Ekeland-main} gives
$$
  \sqrt{\tau(y)}
  \geq\sqrt{T_i}\left(1-\frac1{\sqrt i}\right)=\sqrt{T_i}\sqrt{\sigma_i}.
$$
Thus
\begin{equation}\label{eq:large-positive-cylinder}
  V_i(y,s)>0
  \quad\text{on}\quad
  B_{\sqrt{i}}^{g_i}(x_i)\times[0,\sigma_i).
\end{equation}

Choose $m+1$ distinct numbers
\begin{equation*}\label{eq:fixed-nodes}
  0<a_0<a_1<\cdots<a_m<\frac14.
\end{equation*}
For large $i$, define
$$
  b_i=1-\frac4{\sqrt i}.
$$
Then $b_i\in[1/2,1)$, $b_i<\sigma_i$, and $b_i\to1$ as $i\to\infty$.  The fixed radius required in
Lemma \ref{lem:harnack-package}(1) is contained in
$B_{\sqrt{i}}^{g_i}(x_i)$ for all sufficiently large $i$.  Hence that lemma applies
to \eqref{eq:large-positive-cylinder} and gives constants
$C_0,\ldots,C_m$, independent of $i$, such that
\begin{equation*}\label{eq:node-Harnack}
  0<V_i(x_i,a_\ell)
  \leq C_\ell V_i(x_i,b_i),
  \qquad \ell=0,\ldots,m.
\end{equation*}

Normalize the time polynomial by
\begin{equation*}\label{eq:normalized-pi}
  P_i(s):=\frac{V_i(x_i,s)}{V_i(x_i,b_i)}.
\end{equation*}
Then $P_i$ has degree at most $m$ and satisfies
\begin{equation*}\label{eq:Pi-data}
  0<P_i(a_\ell)\leq C_\ell,
  \qquad
  P_i(b_i)=1,
  \qquad
  P_i(1)=0.
\end{equation*}
Note $P_i$ is determined by its Lagrange interpolating polynomial which passes through the $m+1$ different points $(a_\ell, P(a_\ell))$, $\ell=0, 1, \cdots, m$, i.e. 
\begin{equation*}
 P_i(s)=\sum_{\ell=0}^mP_i(a_\ell)L_\ell(s), \qquad \text{ where }L_\ell(s)=\prod_{j\neq \ell, 0\le j\le m}\frac{s-a_j}{a_\ell-a_j}.   
\end{equation*}
Then for every large $i$, we have $$\sup_{0\le s\le 1}|P_i'(s)|\leq \sum_{\ell=0}^mC_\ell\sup_{0\le s\le 1}|L'_\ell(s)|=:C.$$
Note that the constant $C$ depends only on $m$, $a_\ell$ and $C_\ell$, and is independent of $i$. Therefore
$$
  1=|P_i(b_i)-P_i(1)|
  \leq C(1-b_i)\to0,
$$
a contradiction.  Thus no first zero exists and \eqref{eq:caloric-positive-conclusion} follows.
\end{proof}

\section{Proof of Theorem \ref{thm:main}}\label{sec:polyhar}

We now apply Proposition \ref{thm:caloric-positivity} to the finite Taylor polynomial generated by a polyharmonic function, which leads to the proof of Theorem \ref{thm:main}.

\begin{proof}[Proof of Theorem \ref{thm:main}]
Define
\begin{equation}\label{eq:finite-caloric-poly}
  \mathcal U(x,t)
  =\sum_{j=0}^{k-1}\frac{t^j}{j!}\Delta_g^ju(x).
\end{equation}
Since $\Delta_g^ku=0$, we have 
$$
  \partial_t\mathcal U
  =\sum_{j=0}^{k-2}\frac{t^j}{j!}\Delta_g^{j+1}u
  =\Delta_g\mathcal U.
$$
Also $\mathcal U(x,0)=u(x)>0$. Proposition \ref{thm:caloric-positivity} yields
\begin{equation}\label{eq:U-positive-all-time}
  \mathcal U(x,t)>0
  \qquad\hbox{for $x\in M,\ t\geq0$}.
\end{equation}

For fixed $x$, the function $t\mapsto\mathcal U(x,t)$ is a polynomial which is positive on the whole half-line.  Its coefficient of $t^{k-1}$ cannot be negative.  Hence
\begin{equation*}\label{eq:top-nonnegative}
  \Delta_g^{k-1}u\geq0.
\end{equation*}
Since $\Delta_g(\Delta_g^{k-1}u)=0$, Yau's theorem (Theorem \ref{Yau's theorem}) implies
$$
  \Delta_g^{k-1}u\equiv c_u\geq0.
$$
This is \eqref{constpoly}.

It remains to estimate all iterated Laplacians.  Choose distinct numbers
$$
  0<\theta_0<\theta_1<\cdots<\theta_{k-1}<\Theta.
$$
For fixed $x$ and $T>0$, set
$$
  Q_{x,T}(s)=\mathcal U(x,sT).
$$
The polynomial $Q_{x,T}$ has degree at most $k-1$, and
$$
  Q_{x,T}^{(j)}(0):=\frac{d^j}{ds^j}\bigg|_{s=0}Q_{x, T}=T^j\Delta_g^ju(x).
$$
Differentiating the Lagrange interpolation polynomial with respect to the nodes $\theta_0,\ldots,\theta_{k-1}$ at $s=0$, it gives constants $A_{j\ell}$, which depend only on the chosen nodes, and $j$ and $k$, and are independent of $x$ and $T$, such that
\begin{equation}\label{eq:lagrange-derivatives}
  \Delta_g^ju(x)
  =T^{-j}\sum_{\ell=0}^{k-1}A_{j\ell}\mathcal U(x,\theta_\ell T),
  \qquad 0\leq j\leq k-1.
\end{equation}

Fix the base point $p$ and write $r=d_g(p,x)$.  Take $
  T=1+r^2.$
Since $r^2\leq T$, Lemma~\ref{lem:harnack-package} (2) and \eqref{eq:U-positive-all-time} imply
\begin{equation*}\label{eq:U-two-point}
  \mathcal U(x,\theta_\ell T)
  \leq C\mathcal U(p,\Theta T)
\end{equation*}
for every $\ell$, with $C$ depending only on all fixed times $\theta_\ell$ and $\Theta$, and constants of the doubling property and uniform Neumann--Poincar\'e inequality.  At the fixed point $p$,
\begin{equation*}\label{eq:base-poly-growth}
  0<\mathcal U(p,\Theta T)
  \leq\sum_{j=0}^{k-1}\frac{(\Theta T)^j}{j!}|\Delta_g^ju(p)|
  \leq C_{u,p,k}(1+T^{k-1}).
\end{equation*}
Substitution into \eqref{eq:lagrange-derivatives} gives
\begin{equation}\label{eq:delta-j-v-growth}
  |\Delta_g^ju(x)|
  \leq C_{u,p,j}\,T^{k-1-j}
  \leq C_{u,p,j}(1+r)^{2(k-1-j)}.
\end{equation}
This is exactly the desired estimate \eqref{growth} for $u$. 

Finally, we derive the proof of \eqref{gradgrowth}, and the non-negative Ricci curvature assumption is needed to apply the Li-Yau gradient estimate. By \eqref{eq:U-positive-all-time} and applying Li-Yau gradient estimate \cite{LiYau} to $\mathcal U$, we have,
\begin{equation*}
    \frac{|\nabla \mathcal U|_g^2}{\mathcal U^2}-\frac{\partial_t\mathcal U}{\mathcal U}\le \frac{C_n}{t} \qquad \text{ for every } t>0.
\end{equation*}
Then, by \eqref{eq:finite-caloric-poly} and \eqref{eq:delta-j-v-growth}, for fixed $\theta$, we have $$|\nabla\mathcal{U}(x, \theta T)|_g^2\le \frac{C_n}{\theta T}|\mathcal{U}(x, \theta T)|^2+|\mathcal{U}(x, \theta T)||\partial_t\mathcal U(x, \theta T)|\le C_{\theta, u, p, n}T^{2k-3},$$ and hence 
\begin{equation}\label{gradcaro}
    |\nabla \mathcal U(x, \theta T)|_g\le C_{\theta, u, p, n}T^{k-\frac{3}{2}}.
\end{equation}
Recall $$\nabla \mathcal U(x, \theta_\ell T)=\sum_{j=0}^{k-1}\theta_\ell^j B_j(x, T)\qquad \text{ where } B_j(x, T)=\frac{T^j}{j!}\nabla\Delta_g^ju(x).$$
Note the scalar coefficient matrix $(\theta_\ell^j)_{0\le\ell, j \le k-1}$ is a fixed invertible Vandermonde matrix. Consequently, there are constants $b_{j\ell}$ depending only on the chosen nodes such that 
\begin{equation*}\label{gardB}
    B_j(x, T)=\sum_{\ell=0}^{k-1}b_{j\ell}\nabla \mathcal U(x, \theta_\ell T).
\end{equation*}
Together with \eqref{gradcaro} and $T=1+r^2$, we obtain the desired estimate \eqref{gradgrowth}.
\end{proof}

\section{Proof of Theorem~\ref{thm:local-CY}}\label{sec:Cheng-Yau}
In this section, we provide the proof of Theorem~\ref{thm:local-CY}, which establishes a Cheng--Yau type estimate for positive biharmonic functions on complete manifolds with Ricci curvature bounded from below.

\begin{proof}[Proof of Theorem~\ref{thm:local-CY}] The proof is based on the Bochner technique, and it is divided into the following two steps.

{\sl Step 1. We derive a differential inequality with a coercive term.} Set $f=\log u$ and $q=u^{-1}\Delta_g u$. Then, a direct calculation yields
$$
    \nabla f=\frac{\nabla u}{u}
    \quad\hbox{and}\quad
    \Delta_g f=\frac{\Delta_g u}{u}-|\nabla f|_g^2
    =q-|\nabla f|_g^2.
$$
Using the Bochner formula for $f$, we obtain
\begin{align*}
    \frac12\Delta_g|\nabla f|^2
    &=
    |\nabla^2f|_g^2
    +\langle\nabla f,\nabla\Delta_g f\rangle
    +\Ric_g(\nabla f,\nabla f)\\
    &=
    |\nabla^2f|_g^2
    +\langle\nabla f,\nabla q\rangle
    -\langle\nabla f,\nabla|\nabla f|_g^2\rangle
    +\Ric_g(\nabla f,\nabla f).
\end{align*}
We define the weighted  Laplacian by
$$
    L\cdot=\Delta_g\cdot+2\langle\nabla f,\nabla\cdot\rangle.
$$
It follows that
\begin{equation}
\label{eq:Ls}
    L|\nabla f|^2
    =
    2|\nabla^2f|^2
    +2\Ric_g(\nabla f,\nabla f)
    +2\langle\nabla f,\nabla q\rangle.
\end{equation}
On the other hand, using the biharmonicity of $u$, we obtain
$$
    0=\Delta_g(qu)
    =u\left(\Delta_g q+2\langle\nabla f,\nabla q\rangle+q^2\right),
$$
and so
\begin{equation}
\label{eq:Lq}
    Lq=-q^2.
\end{equation}

Set $\Phi=a|\nabla f|_g^2-q$, where $a>0$ is a constant to be determined. Combining \eqref{eq:Ls} with \eqref{eq:Lq}, we have
\begin{align}
    L\Phi
    &={}
    q^2
    +2a|\nabla^2f|_g^2
    +2a\Ric_g(\nabla f,\nabla f)
    +2a\langle\nabla f,\nabla q\rangle \notag\\
    &={}
    q^2
    +2a|\nabla^2f|_g^2
    +2a\Ric_g(\nabla f,\nabla f)
    +2a^2\langle\nabla f,\nabla|\nabla f|_g^2\rangle
    -2a\langle\nabla f,\nabla\Phi\rangle.       
\label{eq:Lphi-pre}
\end{align}
We note that
$$
    \langle\nabla f,\nabla|\nabla f|_g^2\rangle
    =2\nabla^2f(\nabla f,\nabla f).
$$
Thus, completing the square in \eqref{eq:Lphi-pre} yields
\begin{align*}
    L\Phi+2a\langle\nabla f,\nabla\Phi\rangle
    ={}&
    q^2
    +2a|\nabla^2f+a\,df\otimes df|_g^2
    -2a^3|\nabla f|_g^4\\
    &
    +2a\Ric_g(\nabla f,\nabla f).
\end{align*}
We also note that
$$
    \operatorname{tr}(\nabla^2f+a\,df\otimes df)=\Delta_g f+a|\nabla f|_g^2
    =q-(1-a)|\nabla f|_g^2,
$$
so the curvature assumption and the Cauchy--Schwarz inequality combined yield
\begin{equation}
\label{eq:phi-pre}
\begin{split}
    L\Phi+2a\langle\nabla f,\nabla\Phi\rangle
    \ge{}&
    q^2
    +\frac{2a}{n}
       \bigl(q-(1-a)|\nabla f|_g^2\bigr)^2
    -2a^3|\nabla f|_g^4\\
    &
    -2a(n-1)K|\nabla f|_g^2.
\end{split}
\end{equation}
This inequality determines the choice of $a$ and explains the additional
drift term. We now take $a=1/(8n)$
and define
$$
    L_a\cdot
    =L\cdot+2a\langle\nabla f,\nabla\cdot\rangle
    =\Delta_g\cdot+2(1+a)\langle\nabla f,\nabla\cdot\rangle.
$$
Using
$$
    \bigl(q-(1-a)|\nabla f|_g^2\bigr)^2
    \ge
    \frac12(1-a)^2|\nabla f|_g^4-q^2,
$$
we infer from \eqref{eq:phi-pre} that
\begin{align*}
    L_a\Phi
    \ge{}&
    \left(1-\frac{2a}{n}\right)q^2
    +\left(\frac{a(1-a)^2}{n}-2a^3\right)|\nabla f|_g^4
    -2a(n-1)K|\nabla f|_g^2\\
    \ge{}&
    \frac{47}{512n^2}
       \left(q^2+|\nabla f|_g^4\right)
    -\frac{n-1}{4n}K|\nabla f|_g^2.
\end{align*}
By Young's inequality, we
conclude that
\begin{equation}
\label{eq:Dphi-coercive}
    L_a\Phi
    \ge
    \frac{47}{1024n^2}
       \left(q^2+|\nabla f|_g^4\right)
    -C_nK^2.
\end{equation}

{\sl Step 2. We apply the maximum principle.}
Let $\eta=\chi(r/R)$ be the standard radial Lipschitz cut-off. The estimates below hold away from the cut locus and in the barrier sense at the cut locus.
$$
    0\le\eta\le1 \quad\hbox{in }B_{2R}(p),
    \qquad
    \eta\equiv1\quad\hbox{in }B_R(p),
$$
and
\begin{equation}
\label{eq:cutoff}
    |\nabla\eta|_g
    \le\frac{C_n}{R},
    \qquad
    \Delta_g\eta
    \ge
    -C_n\left(\frac1{R^2}+\frac{\sqrt K}{R}\right).
\end{equation}
Such a cut-off function follows from the Laplacian comparison theorem. It is
smooth away from the cut locus of $p$. 

If $\eta^4\Phi\le 0$ on $B_{2R}(p)$, there is nothing to prove.
Otherwise, $\eta^4\Phi$ attains a positive maximum at some point
$x_0\in B_{2R}(p)$. We may assume that $x_0$ lies away from the
cut locus of $p$. In fact, if $x_0$ lies in the cut locus, the same computation can be carried out in the barrier sense using Calabi's trick; see \cite{Calabi58} and
\cite[Chapter I, Proof of Theorem 3.1]{SY94}.

It is clear that 
$\eta(x_0)>0$, $\Phi(x_0)>0$, and
$$
    \nabla\Phi
    =-4\Phi\frac{\nabla\eta}{\eta}\quad\text{and}\quad\nabla^2(\eta^4\Phi)~\hbox{is negative semidefinite at $x_0$.}
$$
In what follows, all calculations are performed at $x_0$.
Since $L_a(\eta^4\Phi)\le0$, a direct expansion gives
\begin{equation}
\label{eq:max-expansion}
    \eta^4L_a\Phi
    +4\eta^3\Phi\,\Delta_g\eta
    -20\eta^2\Phi|\nabla\eta|_g^2+8(1+a)\eta^3\Phi
       \langle\nabla f,\nabla\eta\rangle\leq0.
\end{equation}
Since $\Phi>0$, combining \eqref{eq:Dphi-coercive},
\eqref{eq:cutoff}, and \eqref{eq:max-expansion}, and using
$\eta\le1$, we obtain
\begin{align*}
    \eta^4\left(q^2+|\nabla f|_g^4\right)
    \le{}&
    C_n\left(\frac1{R^2}+\frac{\sqrt K}{R}\right)\eta^2\Phi
    +\frac{C_n}{R}\eta^3\Phi|\nabla f|_g+C_nK^2\\
    \le{}&
    C_n\left(\frac1{R^2}+\frac{\sqrt K}{R}\right)
    \eta^2\left(q^2+|\nabla f|_g^4\right)^{1/2}
    +\frac{C_n}{R}\left\{\eta^2
    \left(q^2+|\nabla f|_g^4\right)^{1/2}\right\}^{3/2}+C_nK^2.
\end{align*}
Hence, we apply Young's inequality to get 
$$
    \eta^4\left(q^2+|\nabla f|_g^4\right)
    \le
    C_n\left(\frac1{R^4}+K^2\right),
$$
and so
\begin{equation*}
    \eta^2\left(q^2+|\nabla f|_g^4\right)^{1/2}
    \le
    C_n\left(\frac1{R^2}+K\right).
\end{equation*}
Thus, we have 
$$
    \eta^4\Phi(x_0)
    \le
    C_n\eta^4
       \left(q^2+|\nabla f|_g^4\right)^{1/2}(x_0)
    \le
    C_n\left(\frac1{R^2}+K\right).
$$
Since $x_0$ is a point where $\eta^4\Phi$ attains its maximum, and $\eta\equiv1$ on $B_R(p)$, we obtain
$$
    \Phi
    \le
    C_n\left(\frac1{R^2}+K\right)
    \qquad\hbox{in }B_R(p),
$$
which exactly implies 
\eqref{eq:local-main}, and so the estimate 
\eqref{eq:local-gradient-root} follows immediately.
\end{proof}

\section{Proof of Theorem~\ref{thm:sharp-poisson}}\label{sec:sharp}

The final section is devoted to proving Theorem~\ref{thm:sharp-poisson} using \Cref{thm:local-CY,thm:main}.

\begin{proof}[Proof of \Cref{thm:sharp-poisson}]
Since $\Delta_g^2u=0$,  Theorem \ref{thm:main} yields that $\Delta_g u\equiv c$ for some nonnegative constant $c$. Then 
the local
Cheng--Yau type estimate \eqref{eq:local-main} for positive biharmonic functions gives
$$
    -\frac{c}{u}+\frac{|\nabla u|_g^2}{8nu^2}
    \le \frac{C(n)}{R^2}
$$
at any fixed point, where $R>0$ is arbitrary. Sending $R\to\infty$, we obtain
\begin{equation}
\label{eq:coarse-bound}
    |\nabla u|_g^2\le 8ncu.
\end{equation}
If $c=0$, we obtain that $u$ is constant. Therefore, we may assume without loss of generality that $c>0$. Set $v=\sqrt{2c^{-1}u}$ and $w=|\nabla v|_g^2$. 
Then, a direct calculation shows that 
$$\nabla v=\frac{\nabla u}{\sqrt{2cu}}\quad\hbox{and}\quad\Delta_g v=\frac{2\sqrt{2c}u\Delta_g u-\sqrt{2c}|\nabla u|_g^2}{4cu\sqrt u},$$
which implies 
\begin{equation}
\label{eq:v-equation}
    v\Delta_g v+w=1,
\end{equation}
while \eqref{eq:coarse-bound} yields
\begin{equation}
\label{eq:w-bounded}
    0\le w\le4n.
\end{equation}
Using \eqref{eq:v-equation} and the Bochner formula, we obtain
\begin{align*}
    \frac12\Delta_g w
    &=|\nabla^2v|_g^2+\Ric_g(\nabla v,\nabla v)
      +\langle\nabla v,\nabla\Delta_g v\rangle\\
    &=|\nabla^2v|_g^2+\Ric_g(\nabla v,\nabla v)
      -\frac1v\langle\nabla v,\nabla w\rangle
      +\frac{w(w-1)}{v^2}.
\end{align*}
With the curvature assumption, we have 
\begin{equation}
\label{eq:w-differential}
    v^2\Delta_g w+2v\langle\nabla v,\nabla w\rangle
    \ge 2w(w-1).
\end{equation}
Such a differential inequality enlightens us to define the operator
$$
    L_v\cdot=v^2\Delta_g\cdot+2v\langle\nabla v,\nabla\cdot\rangle.
$$

We now apply the barrier maximum principle to prove \eqref{eq:sharp-v}. 
Fix $p\in M$, set
$$
    r(x)=d_g(p,x), \quad \psi(x)=\log(1+r(x)^2).
$$
Then, away from the cut locus of $p$, we have 
\begin{align*}
|\nabla\psi|_g&=\frac{2r}{1+r^2},\\
\Delta_g\psi
    &=\frac{2(1-r^2)}{(1+r^2)^2}
      +\frac{2r}{1+r^2}\Delta_g r
    \le\frac{2n}{1+r^2},
    \end{align*}
    where the Laplacian comparison theorem is used.
We note that \eqref{eq:w-bounded} implies that $|\nabla v|_g\le2\sqrt n$ on $M$. Thus, there holds
$$
    v(x)\le v(p)+2\sqrt n\,r(x)\quad\hbox{for any $x\in M$}.
$$
Therefore, we obtain 
\begin{equation}
\label{eq:psi-bound}
    L_v\psi=v^2\Delta_g\psi+2v\langle\nabla v,\nabla\psi\rangle\le C
\end{equation}
for a constant $C$ independent of $x$. As usual, \eqref{eq:psi-bound} is
understood in the barrier sense at the cut locus by Calabi's trick.

For every $\epsilon>0$, we denote by $x_\epsilon$ the maximum point of the function $w-\epsilon\psi$. Here, we note that the existence of $x_\epsilon$ is ensured by \eqref{eq:w-bounded} and the properness of $\psi$. At $x_\epsilon$, the barrier maximum principle and
\eqref{eq:psi-bound} combined yield
$$
    L_vw(x_\epsilon)
    \le\epsilon L_v\psi(x_\epsilon)
    \le C\epsilon.
$$
It follows from \eqref{eq:w-differential} that
$$
    2w(x_\epsilon)\left(w(x_\epsilon)-1\right)
    \le C\epsilon.
$$
Moreover, we note that $w(x_\epsilon)\to\sup_Mw$ as $\epsilon\to0$. Hence, we obtain 
$$
    \bigl(\sup_Mw\bigr)\bigl(\sup_Mw-1\bigr)\le0,
$$
and so $w\le1$ on $M$, which implies \eqref{eq:sharp-v}. This completes the proof.
\end{proof}

\bibliography{bib}

@article{Calabi58,
 author = {Calabi, Eugenio},
 title = {An extension of {E}. {Hopf}'s maximum principle with an application to {Riemannian} geometry},
 fjournal = {Duke Mathematical Journal},
 journal = {Duke Math. J.},
 issn = {0012-7094},
 volume = {25},
 pages = {45--56},
 year = {1958},
 language = {English},
 doi = {10.1215/S0012-7094-58-02505-5},
 zbMATH = {3129705},
 Zbl = {0079.11801}
}

@book{SY94,
 author = {Schoen, Richard and Yau, Shing-Tung},
 title = {Lectures on differential geometry},
 fseries = {Conference Proceedings and Lecture Notes in Geometry and Topology},
 series = {Conf. Proc. Lect. Notes Geom. Topol.},
 volume = {1},
 isbn = {1-57146-012-8},
 year = {1994},
 publisher = {Cambridge, MA: International Press},
 language = {English},
 zbMATH = {770462},
 Zbl = {0830.53001}
}

@article{Ekeland1974,
 author = {Ekeland, I.},
 title = {On the variational principle},
 fjournal = {Journal of Mathematical Analysis and Applications},
 journal = {J. Math. Anal. Appl.},
 issn = {0022-247X},
 volume = {47},
 pages = {324--353},
 year = {1974},
 language = {English},
 doi = {10.1016/0022-247X(74)90025-0},
 zbMATH = {3449362},
 Zbl = {0286.49015}
}

@article{SaloffCoste1992IMRN,
 author = {Saloff-Coste, Laurent},
 title = {A note on {Poincar{\'e}}, {Sobolev}, and {Harnack} inequalities},
 fjournal = {IMRN. International Mathematics Research Notices},
 journal = {Int. Math. Res. Not.},
 issn = {1073-7928},
 volume = {1992},
 number = {2},
 pages = {27--38},
 year = {1992},
 language = {English},
 doi = {10.1155/S1073792892000047},
 zbMATH = {39647},
 Zbl = {0769.58054}
}

@article{SaloffCoste1992JDG,
 author = {Saloff-Coste, Laurent},
 title = {Uniformly elliptic operators on {Riemannian} manifolds},
 fjournal = {Journal of Differential Geometry},
 journal = {J. Differ. Geom.},
 issn = {0022-040X},
 volume = {36},
 number = {2},
 pages = {417--450},
 year = {1992},
 language = {English},
 doi = {10.4310/jdg/1214448748},
 zbMATH = {4215967},
 Zbl = {0735.58032}
}

@article{Grigoryan1992,
 author = {Grigor'yan, A. A.},
 title = {The heat equation on noncompact {Riemannian} manifolds},
 fjournal = {Mathematics of the USSR, Sbornik},
 journal = {Math. USSR, Sb.},
 issn = {0025-5734},
 volume = {72},
 number = {1},
 year = {1992},
 language = {English},
 doi = {10.1070/SM1992v072n01ABEH001410},
 zbMATH = {419064},
 Zbl = {0776.58035},
 pages={}
}

@article {Yau,
    AUTHOR = {Yau, Shing Tung},
     TITLE = {Harmonic functions on complete {R}iemannian manifolds},
   JOURNAL = {Comm. Pure Appl. Math.},
  FJOURNAL = {Communications on Pure and Applied Mathematics},
    VOLUME = {28},
      YEAR = {1975},
     PAGES = {201--228},
      ISSN = {0010-3640,1097-0312},
   MRCLASS = {53C20 (31C05)},
  MRNUMBER = {431040},
MRREVIEWER = {Yoshiaki\ Maeda},
       DOI = {10.1002/cpa.3160280203},
       URL = {https://doi.org/10.1002/cpa.3160280203},
}

@article {Kuran,
    AUTHOR = {Kuran, \"U.},
     TITLE = {Generalizations of a theorem on harmonic functions},
   JOURNAL = {J. London Math. Soc.},
  FJOURNAL = {The Journal of the London Mathematical Society},
    VOLUME = {41},
      YEAR = {1966},
     PAGES = {145--152},
      ISSN = {0024-6107,1469-7750},
   MRCLASS = {31.11},
  MRNUMBER = {192071},
MRREVIEWER = {B.\ H.\ Murdoch},
       DOI = {10.1112/jlms/s1-41.1.145},
       URL = {https://doi.org/10.1112/jlms/s1-41.1.145},
}

@article {Li-Tam,
    AUTHOR = {Li, Peter and Tam, Luen-Fai},
     TITLE = {Complete surfaces with finite total curvature},
   JOURNAL = {J. Differential Geom.},
  FJOURNAL = {Journal of Differential Geometry},
    VOLUME = {33},
      YEAR = {1991},
    NUMBER = {1},
     PAGES = {139--168},
      ISSN = {0022-040X,1945-743X},
   MRCLASS = {53C21 (53C22 53C23 53C45)},
  MRNUMBER = {1085138},
MRREVIEWER = {Chun-Li\ Shen},
       URL = {http://projecteuclid.org/euclid.jdg/1214446033},
}

@article {CM97,
    AUTHOR = {Colding, Tobias H. and Minicozzi, II, William P.},
     TITLE = {Harmonic functions on manifolds},
   JOURNAL = {Ann. of Math. (2)},
  FJOURNAL = {Annals of Mathematics. Second Series},
    VOLUME = {146},
      YEAR = {1997},
    NUMBER = {3},
     PAGES = {725--747},
      ISSN = {0003-486X,1939-8980},
   MRCLASS = {53C21 (31B05 58G30)},
  MRNUMBER = {1491451},
MRREVIEWER = {Tanya\ J.\ Christiansen},
       DOI = {10.2307/2952459},
       URL = {https://doi.org/10.2307/2952459},
}

@article {Cheng-Yau,
    AUTHOR = {Cheng, S. Y. and Yau, S. T.},
     TITLE = {Differential equations on {R}iemannian manifolds and their
              geometric applications},
   JOURNAL = {Comm. Pure Appl. Math.},
  FJOURNAL = {Communications on Pure and Applied Mathematics},
    VOLUME = {28},
      YEAR = {1975},
    NUMBER = {3},
     PAGES = {333--354},
      ISSN = {0010-3640,1097-0312},
   MRCLASS = {53C20 (58G99)},
  MRNUMBER = {385749},
MRREVIEWER = {Lung\ Ock\ Chung},
       DOI = {10.1002/cpa.3160280303},
       URL = {https://doi.org/10.1002/cpa.3160280303},
}

@article {Martinazzi,
    AUTHOR = {Martinazzi, Luca},
     TITLE = {Classification of solutions to the higher order {L}iouville's
              equation on {$\mathbb{R}^{2m}$}},
   JOURNAL = {Math. Z.},
  FJOURNAL = {Mathematische Zeitschrift},
    VOLUME = {263},
      YEAR = {2009},
    NUMBER = {2},
     PAGES = {307--329},
      ISSN = {0025-5874,1432-1823},
   MRCLASS = {53C21},
  MRNUMBER = {2534120},
MRREVIEWER = {Niels\ Martin\ M\o ller},
       DOI = {10.1007/s00209-008-0419-1},
       URL = {https://doi.org/10.1007/s00209-008-0419-1},
}

@article {Li,
    AUTHOR = {Li, Mingxiang},
     TITLE = {The total {$Q$}-curvature, volume entropy and polynomial growth
              polyharmonic functions},
   JOURNAL = {Adv. Math.},
  FJOURNAL = {Advances in Mathematics},
    VOLUME = {450},
      YEAR = {2024},
     PAGES = {Paper No. 109768, 43},
      ISSN = {0001-8708,1090-2082},
   MRCLASS = {53C18 (53C20 58J90)},
  MRNUMBER = {4754948},
MRREVIEWER = {Weihong\ Xie},
       DOI = {10.1016/j.aim.2024.109768},
       URL = {https://doi.org/10.1016/j.aim.2024.109768},
}

@book {GGS,
    AUTHOR = {Gazzola, Filippo and Grunau, Hans-Christoph and Sweers, Guido},
     TITLE = {Polyharmonic boundary value problems},
    SERIES = {Lecture Notes in Mathematics},
    VOLUME = {1991},
      NOTE = {Positivity preserving and nonlinear higher order elliptic
              equations in bounded domains},
 PUBLISHER = {Springer-Verlag, Berlin},
      YEAR = {2010},
     PAGES = {xviii+423},
      ISBN = {978-3-642-12244-6},
   MRCLASS = {35-02 (31B30 35A08 35J40 35J61 46E35 58E12)},
  MRNUMBER = {2667016},
MRREVIEWER = {Rodney\ Josu\'e\ Biezuner},
       DOI = {10.1007/978-3-642-12245-3},
       URL = {https://doi.org/10.1007/978-3-642-12245-3},
}

@article {Gover,
    AUTHOR = {Gover, A. R.},
     TITLE = {Laplacian operators and {$Q$}-curvature on conformally
              {E}instein manifolds},
   JOURNAL = {Math. Ann.},
  FJOURNAL = {Mathematische Annalen},
    VOLUME = {336},
      YEAR = {2006},
    NUMBER = {2},
     PAGES = {311--334},
      ISSN = {0025-5831,1432-1807},
   MRCLASS = {58J60 (53C20 53C21)},
  MRNUMBER = {2244375},
MRREVIEWER = {Mohameden\ Ahmedou},
       DOI = {10.1007/s00208-006-0004-z},
       URL = {https://doi.org/10.1007/s00208-006-0004-z},
}

@article {CG,
    AUTHOR = {Case, Jeffrey S. and Gover, A. Rod},
     TITLE = {The {GJMS} operators in geometry, analysis and physics},
   JOURNAL = {J. Lond. Math. Soc. (2)},
  FJOURNAL = {Journal of the London Mathematical Society. Second Series},
    VOLUME = {113},
      YEAR = {2026},
    NUMBER = {1},
     PAGES = {Paper No. e70375, 21},
      ISSN = {0024-6107,1469-7750},
   MRCLASS = {53C18 (32V05 35Q40 53C21 58J70)},
  MRNUMBER = {5011581},
       DOI = {10.1112/jlms.70375},
       URL = {https://doi.org/10.1112/jlms.70375},
}

@article{BC1,
  title={Liouville theorem for biharmonic functions on manifolds of nonnegative {Ricci} curvature},
  author={Bravo, John E and Cortissoz, Jean C},
  journal={preprint, arXiv:2511.08358},
  year={2025},
  pages={}
}

@article{BC2,
  title={A Polyharmonic Liouville Hierarchy on Complete Manifolds of Nonnegative {Ricci} Curvature},
  author={Bravo, John E and Cortissoz, Jean C},
  journal={preprint, arXiv:2512.04141},
  year={2025},
  pages={}
}

@article{WZ,
  title={The qualitative behavior for biharmonic functions on open manifolds},
  author={Wang, Lin and Zhu, Miaomiao},
  journal={preprint, arXiv:2511.09393},
  year={2025}
}

@article {CCM,
    AUTHOR = {Cheeger, J. and Colding, T. H. and Minicozzi, II, W. P.},
     TITLE = {Linear growth harmonic functions on complete manifolds with
              nonnegative {R}icci curvature},
   JOURNAL = {Geom. Funct. Anal.},
  FJOURNAL = {Geometric and Functional Analysis},
    VOLUME = {5},
      YEAR = {1995},
    NUMBER = {6},
     PAGES = {948--954},
      ISSN = {1016-443X,1420-8970},
   MRCLASS = {53C21 (58G30)},
  MRNUMBER = {1361516},
MRREVIEWER = {Man\ Chun\ Leung},
       DOI = {10.1007/BF01902216},
       URL = {https://doi.org/10.1007/BF01902216},
}

@article {CM-CPAM,
    AUTHOR = {Colding, Tobias H. and Minicozzi, II, William P.},
     TITLE = {Liouville theorems for harmonic sections and applications},
   JOURNAL = {Comm. Pure Appl. Math.},
  FJOURNAL = {Communications on Pure and Applied Mathematics},
    VOLUME = {51},
      YEAR = {1998},
    NUMBER = {2},
     PAGES = {113--138},
      ISSN = {0010-3640,1097-0312},
   MRCLASS = {53C21 (53C20 58E15)},
  MRNUMBER = {1488297},
MRREVIEWER = {Man\ Chun\ Leung},
       DOI = {10.1002/(SICI)1097-0312(199802)51:2<113::AID-CPA1>3.0.CO;2-E},
       URL =
              {https://doi.org/10.1002/(SICI)1097-0312(199802)51:2<113::AID-CPA1>3.0.CO;2-E},
}

@article {Li-Tam-92-JDG,
    AUTHOR = {Li, Peter and Tam, Luen-Fai},
     TITLE = {Harmonic functions and the structure of complete manifolds},
   JOURNAL = {J. Differential Geom.},
  FJOURNAL = {Journal of Differential Geometry},
    VOLUME = {35},
      YEAR = {1992},
    NUMBER = {2},
     PAGES = {359--383},
      ISSN = {0022-040X,1945-743X},
   MRCLASS = {53C21 (53C20 58G30)},
  MRNUMBER = {1158340},
MRREVIEWER = {Yang\ Lian\ Pan},
       URL = {http://projecteuclid.org/euclid.jdg/1214448079},
}

@article {linzh2019,
    AUTHOR = {Lin, Fanghua and Zhang, Q. S.},
     TITLE = {On ancient solutions of the heat equation},
   JOURNAL = {Comm. Pure Appl. Math.},
  FJOURNAL = {Communications on Pure and Applied Mathematics},
    VOLUME = {72},
      YEAR = {2019},
    NUMBER = {9},
     PAGES = {2006--2028},
      ISSN = {0010-3640,1097-0312},
   MRCLASS = {35K05 (35K08 58J35)},
  MRNUMBER = {3987724},
MRREVIEWER = {Juan\ C.\ Pozo},
       DOI = {10.1002/cpa.21820},
       URL = {https://doi.org/10.1002/cpa.21820},
}

@article {coldmin2021,
    AUTHOR = {Colding, Tobias Holck and Minicozzi, II, William P.},
     TITLE = {Optimal bounds for ancient caloric functions},
   JOURNAL = {Duke Math. J.},
  FJOURNAL = {Duke Mathematical Journal},
    VOLUME = {170},
      YEAR = {2021},
    NUMBER = {18},
     PAGES = {4171--4182},
      ISSN = {0012-7094,1547-7398},
   MRCLASS = {53C21},
  MRNUMBER = {4348235},
MRREVIEWER = {Jui-Tang\ Chen},
       DOI = {10.1215/00127094-2021-0015},
       URL = {https://doi.org/10.1215/00127094-2021-0015},
}

@article {rowi1959,
    AUTHOR = {Rosenbloom, P. C. and Widder, D. V.},
     TITLE = {Expansions in terms of heat polynomials and associated
              functions},
   JOURNAL = {Trans. Amer. Math. Soc.},
  FJOURNAL = {Transactions of the American Mathematical Society},
    VOLUME = {92},
      YEAR = {1959},
     PAGES = {220--266},
      ISSN = {0002-9947,1088-6850},
   MRCLASS = {41.00 (35.00)},
  MRNUMBER = {107118},
MRREVIEWER = {J.\ Blackman},
       DOI = {10.2307/1993155},
       URL = {https://doi.org/10.2307/1993155},
}

@article {Widder,
    AUTHOR = {Widder, D. V.},
     TITLE = {The role of the {A}ppell transformation in the theory of heat
              conduction},
   JOURNAL = {Trans. Amer. Math. Soc.},
  FJOURNAL = {Transactions of the American Mathematical Society},
    VOLUME = {109},
      YEAR = {1963},
     PAGES = {121--134},
      ISSN = {0002-9947,1088-6850},
   MRCLASS = {44.30},
  MRNUMBER = {154068},
MRREVIEWER = {I.\ I.\ Hirschman, Jr.},
       DOI = {10.2307/1993650},
       URL = {https://doi.org/10.2307/1993650},
}

@article {ChangWangYang,
    AUTHOR = {Chang, Sun-Yung A. and Wang, Lihe and Yang, Paul C.},
     TITLE = {A regularity theory of biharmonic maps},
   JOURNAL = {Comm. Pure Appl. Math.},
  FJOURNAL = {Communications on Pure and Applied Mathematics},
    VOLUME = {52},
      YEAR = {1999},
    NUMBER = {9},
     PAGES = {1113--1137},
      ISSN = {0010-3640,1097-0312},
   MRCLASS = {58E20},
  MRNUMBER = {1692148},
MRREVIEWER = {Giandomenico\ Orlandi},
       DOI = {10.1002/(SICI)1097-0312(199909)52:9<1113::AID-CPA4>3.0.CO;2-7},
       URL =
              {https://doi.org/10.1002/(SICI)1097-0312(199909)52:9<1113::AID-CPA4>3.0.CO;2-7},
}

@article {MonOni,
    AUTHOR = {Montaldo, S. and Oniciuc, C.},
     TITLE = {A short survey on biharmonic maps between {R}iemannian
              manifolds},
   JOURNAL = {Rev. Un. Mat. Argentina},
  FJOURNAL = {Revista de la Uni\'on Matem\'atica Argentina},
    VOLUME = {47},
      YEAR = {2006},
    NUMBER = {2},
     PAGES = {1--22},
      ISSN = {0041-6932,1669-9637},
   MRCLASS = {53C43 (31B30)},
  MRNUMBER = {2301373},
MRREVIEWER = {John\ C.\ Wood},
}

@article {Wang,
    AUTHOR = {Wang, Changyou},
     TITLE = {Stationary biharmonic maps from {$\mathbb{R}^m$} into a
              {R}iemannian manifold},
   JOURNAL = {Comm. Pure Appl. Math.},
  FJOURNAL = {Communications on Pure and Applied Mathematics},
    VOLUME = {57},
      YEAR = {2004},
    NUMBER = {4},
     PAGES = {419--444},
      ISSN = {0010-3640,1097-0312},
   MRCLASS = {58E20},
  MRNUMBER = {2026177},
MRREVIEWER = {Giandomenico\ Orlandi},
       DOI = {10.1002/cpa.3045},
       URL = {https://doi.org/10.1002/cpa.3045},
}

@article {Struwe,
    AUTHOR = {Struwe, Michael},
     TITLE = {Partial regularity for biharmonic maps, revisited},
   JOURNAL = {Calc. Var. Partial Differential Equations},
  FJOURNAL = {Calculus of Variations and Partial Differential Equations},
    VOLUME = {33},
      YEAR = {2008},
    NUMBER = {2},
     PAGES = {249--262},
      ISSN = {0944-2669,1432-0835},
   MRCLASS = {35J40 (35B65 35J55 58E20)},
  MRNUMBER = {2413109},
MRREVIEWER = {Martin\ Fuchs},
       DOI = {10.1007/s00526-008-0175-4},
       URL = {https://doi.org/10.1007/s00526-008-0175-4},
}

@article {Y3,
    AUTHOR = {Yau, Shing-Tung},
     TITLE = {Nonlinear analysis in geometry},
   JOURNAL = {Enseign. Math. (2)},
  FJOURNAL = {L'Enseignement Math\'{e}matique. Revue Internationale. 2e
              S\'{e}rie},
    VOLUME = {33},
      YEAR = {1987},
    NUMBER = {1-2},
     PAGES = {109--158},
      ISSN = {0013-8584},
   MRCLASS = {58-02 (58G30)},
  MRNUMBER = {896385},
}

@article {Y4,
    AUTHOR = {Yau, Shing-Tung},
 TITLE = {Differential geometry: partial differential equations on
              manifolds},
    Journal = {Proc. Sympos. Pure Math.},
    VOLUME = {54},
     PAGES = {Part I},
 PUBLISHER = {Amer. Math. Soc., Providence, RI},
      YEAR = {1993},
      ISBN = {0-8218-1494-X},
   MRCLASS = {53-02},
  MRNUMBER = {1216573},
       DOI = {10.1090/pspum/054.1/1216573},
       URL = {https://doi.org/10.1090/pspum/054.1/1216573},
}

@article {Y5,
    AUTHOR = {Yau, Shing-Tung},
     TITLE = {Chern---a great geometer of the twentieth century},
     PAGES = {275--319},
 Journal = {Int. Press, Hong Kong},
      YEAR = {1992},
      ISBN = {962-7670-02-2},
   MRCLASS = {01A70},
  MRNUMBER = {1201368},
       DOI = {10.1145/143164.143364},
       URL = {https://doi.org/10.1145/143164.143364},
}

@incollection {L1,
    AUTHOR = {Li, Peter},
     TITLE = {The theory of harmonic functions and its relation to geometry},
 BOOKTITLE = {Differential geometry: partial differential equations on
              manifolds},
    SERIES = {Proc. Sympos. Pure Math.},
    VOLUME = {54},
     PAGES = {307--315},
 PUBLISHER = {Amer. Math. Soc., Providence, RI},
      YEAR = {1993},
      ISBN = {0-8218-1494-X},
   MRCLASS = {53C20 (58G30)},
  MRNUMBER = {1216591},
       DOI = {10.1090/pspum/054.1/1216591},
       URL = {https://doi.org/10.1090/pspum/054.1/1216591},
}

@article {DF,
    AUTHOR = {Donnelly, Harold and Fefferman, Charles},
     TITLE = {Nodal domains and growth of harmonic functions on noncompact
              manifolds},
   JOURNAL = {J. Geom. Anal.},
  FJOURNAL = {The Journal of Geometric Analysis},
    VOLUME = {2},
      YEAR = {1992},
    NUMBER = {1},
     PAGES = {79--93},
      ISSN = {1050-6926,1559-002X},
   MRCLASS = {58G25 (53C21)},
  MRNUMBER = {1140898},
MRREVIEWER = {Robert\ Brooks},
       DOI = {10.1007/BF02921335},
       URL = {https://doi.org/10.1007/BF02921335},
}

@article {LiYau,
    AUTHOR = {Li, Peter and Yau, Shing-Tung},
     TITLE = {On the parabolic kernel of the {S}chr\"odinger operator},
   JOURNAL = {Acta Math.},
  FJOURNAL = {Acta Mathematica},
    VOLUME = {156},
      YEAR = {1986},
    NUMBER = {3-4},
     PAGES = {153--201},
      ISSN = {0001-5962,1871-2509},
   MRCLASS = {58G11 (35J10)},
  MRNUMBER = {834612},
MRREVIEWER = {Harold\ Donnelly},
       DOI = {10.1007/BF02399203},
       URL = {https://doi.org/10.1007/BF02399203},
}
\bibliographystyle{plain}

\end{document}